\documentclass[a4paper,10pt]{amsart}
\usepackage{amssymb, amsmath, graphicx, amsthm, enumitem, multicol,amsfonts}
\usepackage[margin=1in]{geometry}
\usepackage{cite}
\usepackage{ bbold }
\usepackage{longtable}
\usepackage{booktabs}
\usepackage{placeins}
\usepackage{hyperref}

\usepackage[utf8]{inputenc}

\DeclareMathOperator{\C}{\mathbb{C}}
\DeclareMathOperator{\Z}{\mathbb{Z}}
\DeclareMathOperator{\N}{\mathbb{N}}
\newcommand{\Vorb}{{V^{orb(G)}}}

\newcommand{\ba}{\begin{align}}
\newcommand{\ea}{\end{align}}
\newcommand{\bea}{\begin{eqnarray}}
\newcommand{\eea}{\end{eqnarray}}
\newcommand{\be}{\begin{equation}}
\newcommand{\ee}{\end{equation}}

\newcounter{saveenumi}

\title{Orbifolds of Lattice Vertex Operator Algebras at $d=48$ and $d=72$}
\author{Thomas Gem\"unden}
\address{Thomas Gem\"unden, Department of Mathematics, ETH Zurich
	CH-8092 Zurich, Switzerland}
\email{thomas.gemuenden@math.ethz.ch}
\author{ Christoph A.~Keller}
\address{Christoph A. Keller, Department of Mathematics, ETH Zurich
	CH-8092 Zurich, Switzerland and Department of Mathematics, University of Arizona, Tucson, AZ 85721-0089, USA}
\email{christoph.keller@math.ethz.ch}

\theoremstyle{plain}
\newtheorem{thm}{Theorem}[section]
\newtheorem{lem}[thm]{Lemma}
\newtheorem{prop}[thm]{Proposition}
\newtheorem{cor}[thm]{Corollary}

\theoremstyle{definition}
\newtheorem{defn}{Definition}[section]

\theoremstyle{remark}
\newtheorem*{rem}{Remark}
\newtheorem*{note}{Note}

\begin{document}

\begin{abstract}
Motivated by the notion of extremal vertex operator algebras, we investigate cyclic orbifolds of vertex operator algebras coming from extremal even self-dual lattices in $d=48$ and $d=72$.
In this way we construct about one hundred new examples of holomorphic VOAs with a small number of low weight states.
\end{abstract}

\maketitle

\section{Introduction}
An even unimodular lattice of dimension $d$ is called extremal if it has no non-trivial vectors with length squared shorter than $2+2 \lfloor \frac{d}{24}\rfloor$ \cite{MR0376536}. 
This motivates the definition of an analogous notion for holomorphic vertex operator algebras (VOA): A holomorphic VOA of central charge $c$ is extremal if the only states of conformal
weight less than $2+2 \lfloor \frac{c}{24}\rfloor$ are Virasoro descendants \cite{MR2388095}. 
There is also a physical motivation for this definition coming from AdS/CFT holography and pure gravity \cite{Witten:2007kt}.

For $d=24$ the unique extremal lattice is the Leech lattice, and for $c=24$ an extremal VOA is the monster VOA $V^\natural$.
For $d=48$ and $d=72$, several examples of extremal lattices have been constructed \cite{MR1662447,MR1489922,MR2999133,MR3225314}.
There are however no known examples of extremal VOAs of these central charges, and their existence is an open question.
In this article we are interested in a related question: We want to construct holomorphic VOAs which are not necessarily extremal, but which have a relatively small number of states of low weight.
We focus on central charge $48$ and $72$ here. In that case an extremal VOA would have
$$
\dim V_{(1)} =0 \ , \qquad \dim V_{(2)}= 1\ ,  \qquad \textrm{(and for } c=72: ) \qquad \dim V_{(3)}=1\ .
$$
Our goal is to construct examples which come as close as possible to these numbers. We find in fact
\begin{thm}\label{thmmain}
There exists  a tame holomorphic VOA with central charge $c=48$ and
$$
\dim V_{(1)} =0 \ , \qquad \dim V_{(2)}= 48\ .
$$
There also exists such a VOA with $c=72$ and
$$
\dim V_{(1)} =0 \ , \qquad \dim V_{(2)}= 36\ , \qquad \dim V_{(3)}=408\ .
$$
\end{thm}
A natural idea to construct an extremal VOA is of course to start with extremal lattices and construct the corresponding lattice VOAs.
These, however, are still far from extremal VOAs, since the Heisenberg modes give a large number of states in $V_{(1)}$ and $V_{(2)}$.
To reduce this number of states, we want to orbifold by subgroups of  automorphism group of the VOA, thereby eliminating non-invariant states in exchange for adding states from twisted modules.
As we will see, in a wide range of cases this leads to a significant net reduction in the number of low weight states.

Let us briefly review the status of orbifold VOAs with a focus on holomorphic VOAs.
We call a VOA $V$ \emph{tame} if it is rational, $C_2$-cofinite, simple, self-contragredient and of CFT-type. 
The significance of these assumption is that it was proven by Huang \cite{MR2387861} that the fusion rules for the modules of a tame VOA satisfy the Verlinde formula \cite{Verlinde:1988sn} and hence that the modules
form a modular tensor category \cite{MR2468370}.
A rational VOA which only has itself as an irreducible module is called holomorphic.
Lattice VOAs corresponding to even, unimodular lattices are the best known class of tame, holomorphic VOAs and the one we are primarily interested in.
To construct an orbifold of a VOA $V$, let us first pick a subgroup $G$ of the automorphism group of $V$. In a first step one can construct the modules of the fixed-point VOA $V^G$.
By the combined results of \cite{MR1684904,MR3320313,2016arXiv160305645C}, if $G$ is a solvable group and $V$ is tame, then $V^G$ is again tame.
The modular tensor category $V^G$-mod of its modules can be obtained as the module category of a twisted Drinfeld double of the group $G$, $D^\omega(G)$-mod.
Here $\omega$ is a 3-cocycle $\omega \in H^3(G,U(1))$.
This was first discussed in \cite{Roche:1990hs}, following up on work on the operator algebra of general orbifolds \cite{Dijkgraaf:1989hb}.
Together with $G$, $\omega$ thus completely determines the fusion rules and the $S$ matrices of the modular tensor category. 
In general $V^G$ will not be holomorphic. To construct new holomorphic VOAs, we therefore need to find holomorphic extensions of $V^G$.
That is, we want to adjoin modules to $V^G$ such that we recover a holomorphic VOA $V^{orb(G)}$. The general theory of such extensions is described in \cite{EvansGannon}.
The data of a holomorphic extension is a subgroup $H \leq G$ such that $\omega|_H$ is trivial,
and a choice of 2-cocycle $\psi \in Z^2(H\times H, U(1))$ (`discrete torsion') such that $\psi((h,h),(h_1,h_2))=1$ and $\psi((h,1),(1,h'))=\psi((1,h'),(h,1))$.

In this article we will concentrate on cyclic orbifolds only. This case was fully described in \cite{1507.08142}, and one does not need to make use of the general construction described above.
In particular, the characters of all modules can be obtained from the untwisted sector by modular transformations,
so that we do not need to construct explicitly the twisted modules and the action of $g$ on them. For the extension problem, 
we can use the results obtained in \cite{1507.08142} which guarantee the existence of a holomorphic extension of $V^G$ under the right conditions on $H$,
namely that it is of type 0, or (in physics language) that level-matching is satisfied.

Using this technology, we systematically investigate cyclic orbifolds of the four known extremal lattices in $d=48$ and the one known extremal lattice in $d=72$. In doing so we construct around a hundred new holomorphic VOAs.
As stated in Theorem~\ref{thmmain}, we construct an example with $c=48$ and $\dim V_{(1)}=0, \dim V_{(2)}=48$.
This is of course no extremal VOA, but it is a vast improvement over the unorbifolded lattice VOA which has $\dim V_{(2)}=1224$. For $d=72$ we find an example with $\dim V_{(2)}=36$ and $\dim V_{(3)}=408$.

Let us note that a large number of our new VOAs have no spin-one currents, that is $\dim V_{(1)}=0$. This suggests that their automorphism groups may be finite just as in the case of $V^\natural$. It may be interesting therefore to search for moonshine in those examples.

The bulk of our work involves computing the characters of all such orbifolds.
We automatized this process by using Hecke-Schoenberg for the lattice theta functions and the fact that we were able to express all modular forms that appeared as quotients of eta functions.
This reduces the computational work to computing theta functions of fixed point lattices, which we did using MAGMA \cite{MR1484478}.
For both 48- and 72-dimensional lattices we could cover all but a small number of cyclic automorphism groups in the sense that we found at least one lift where this was possible.
The remaining cases require computations of lattice theta functions that we were unable to complete within acceptable time.

\emph{Acknowledgements:}
	We thank Patrick Hnilicka and  Gabriele Nebe for helpful discussions. We are very grateful to Sven M\"oller and Nils Scheithauer for helpful discussions, sharing some of their computer code with us and for very helpful remarks on our draft.
	We particularly thank David Evans and Terry Gannon for sharing their draft with us and for explaining their results to us. CAK thanks the Banff International Research Station for hospitality.
	TG and CAK are supported by the Swiss National Science Foundation through the NCCR SwissMAP. The work of TG is supported by the Swiss National Science Foundation Project Grant 175494.

\section{Cyclic Orbifolds}
In this section we will collect results on cyclic orbifolds, mainly from
\cite{1507.08142} and \cite{2016arXiv161109843M}. 
For simplicity we will always assume that the VOA $V$ is holomorphic and tame,
as in the main part of this article we will only be interested in lattice VOAs corresponding to even, unimodular lattices which are indeed tame and holomorphic,
even though some of the results we quote also hold under weaker assumptions on $V$.

Let  $g$ be an automorphism of $V$ of order $n$. Then $V$ decomposes into $g$-eigenspaces
\[
 V := \bigoplus_{r\in\Z_n} V^r,
\]
where  $V^r = \{v\in V \mid gv = e(\frac{r}{n})v\}$ and $e(q)$ denotes $e^{2\pi i q}$.

\begin{thm}{\cite{MR1794264,1507.08142}}
 Up to isomorphism $V$ possesses a unique irreducible $g$-twisted $V$-module $V(g)$
 and the conformal weight $\rho_g$ of $V(g)$ lies in $(1/n^2)\Z$.
\end{thm}

Let $G = \langle g \rangle$ be a finite, cyclic group of automorphisms of $V$ of order $n$. Then for each $h\in G$ there is a representation of $G$
$\phi_h: G \to \text{Aut}_{\C}(V(h))$ on the vector space of $V(h)$ such that
\[
 \phi_h(k)Y_{V(h)}(v,x)\phi_h(k)^{-1} = Y_{V(h)}(kv,x),
\]
for all $k \in G$ and $v \in V$ and these representations are unique up to multiplication by an $n$-th root of unity.

We decompose $V(g^j)$ in $\phi_j(g)$ eigenspaces
\[
 V(g^j) = \bigoplus_{l\in \Z_n} W^{(j,l)},
\]
where
\[
 W^{(j,l)} = \{w\in V(g^j) \mid \phi_j(g)v = e(l/n)v\}.
\]

Combining results of \cite{2016arXiv160305645C,MR3320313,MR2040864} we have
\begin{thm}
 The fixed-point subalgebra $V^G$ is again a tame VOA
 and up to isomorphism there are exactly $n^2$ distinct irreducible $V^G$-modules, namely the eigenspaces $W^{(i,j)}$.
\end{thm}

\begin{defn}[Twisted trace function]
 The \textit{twisted trace functions} are defined as 
 \[
  T(v,i,j,\tau) = \text{tr}\mid_{V(g^i)}o(v)\phi_i(g^j)q^{L_0-c/24}
 \]
where $o(v) = v_{\text{wt}(v)-1}$ for homogeneous $v$, linearly extended to $V$.
\end{defn}

Note that
\begin{equation}\label{eq:char}
 T_{W^{(i,j)}}(v,\tau) = \frac{1}{n} \sum_{k \in \Z_n} e(-jk/n)T(v,i,k,\tau)
\end{equation}

For later convenience we define
\begin{equation}\label{eq:twtr}
 T(i,j,\tau) = T(\mathbb{1},i,j,\tau) = \text{tr}\mid_{V(g^i)}\phi_i(g^j)q^{L_0-c/24}
\end{equation}
The function $T(i,j,\tau)$ is also called a \textit{twisted character} for the automorphism $g^j$ on the twisted module $V(g^i)$.

\begin{defn}[Type]
 We define the \textit{type} $t$ of an vertex operator algebra automorphism $g$ by
 \[
  t=n^2 \rho_g \mod{n}
 \]
\end{defn}

The modular transformation properties of the twisted characters depend on the choice of the representations $\phi_i$.
It was shown in \cite{1507.08142} and \cite{2016arXiv161109843M} that the they can chosen such that the following results hold:
\begin{thm}[Modular Invariance of Trace Functions]\label{thm:mod}
Let $V$ have central charge $c$ such that $8\mid c$ and $g$ be an automorphism of order $n$ and type $0$. Then the twisted traces transform under the action
of a modular transformation $M = \begin{pmatrix} \alpha & \beta \\ \gamma & \delta \end{pmatrix} \in \mathrm{SL}_2(\Z)$ as
\[
 (\gamma\tau + \delta)^{-\text{wt}(v)}T(v,i,j,M\cdot\tau) = Z(M)T(v,(i,j)\cdot M, \tau),
\]
where the character $Z(M): \mathrm{SL}_2(\Z) \to U_3$ is given by
\[
 Z(M) =
 \begin{cases}
  e((-c/24)(\beta-\gamma)\delta) &\text{if } 3 \nmid \delta \\
  e((-c/24)(\beta+(\alpha+1)\gamma) &\text{if } 3 \mid \delta
 \end{cases}
\]
In particular, if $24 \mid c$ and $v = \mathbb{1}$ we find that
\[
 T(\mathbb{1},i,j,M\cdot\tau) = T(\mathbb{1},(i,j)\cdot M, \tau)
\]
\end{thm}
\begin{cor}\label{cor:mod}
 For $V$ and $g$ as in the theorem above we find that
 \begin{enumerate}
  \item $T(v,i,j,\tau)$ and $T_{W^{(i,j)}}(v,\tau)$, $i,j \in \Z_n$, are modular forms for $\Gamma(n)$,
  \item $T(v,0,j,\tau)$, $j \in \Z_n$, are modular forms for $\Gamma_1(n)$,
  \item $T(v,0,0,\tau)$ is a modular form for $\Gamma$,
  \item $T_{W^{(0,0)}}(v,\tau)$ is a modular form for $\Gamma_0(n)$.
 \end{enumerate}
\end{cor}
Here we denote $\mathrm{SL}_2(\Z)$ by $\Gamma$, and
$\Gamma(n),\Gamma_1(n)$ and $\Gamma_0(n)$ denote the usual congruence subgroups.

We now want to adjoin a suitable set of twisted modules to the fixed point algebra $V^G$ to obtain a holomorphic VOA $V^{orb(G)}$.
To this end let us state the central theorem allowing us to extend $V^G$ to a  holomorphic VOA:
\begin{thm}[Cyclic orbifold \cite{1507.08142}]
	Let $G = \langle g \rangle$ be a finite, cyclic group of automorphisms of $V$ of order $n$ and type $0$. Then the direct sum
	\[
	V^{orb(G)} = \bigoplus_{i \in \Z_n}W^{(i,0)}
	\]
	admits the structure of a holomorphic, rational vertex operator algebra of CFT-Type extending $V^G$ and its irreducible modules. The vertex operator algebra $V^{orb(G)}$ will be called the standard orbifold of $V$ by $G$.
\end{thm}
Note that the original theorem in \cite{1507.08142} is more general. For obtaining holomorphic VOAs, our version is general enough.

Using equation \ref{eq:char} we find that the character $\Vorb$ is given by
\begin{equation}\label{eq:orbchar}
  \text{ch}_{\Vorb}(\tau) = \sum_{i \in \Z_n}T_{W^{(i,0)}}(\mathbb{1},\tau) = \frac{1}{n} \sum_{i,j \in \Z_n} T(\mathbb{1},i,j,\tau) = \frac{1}{n} \sum_{i,j \in \Z_n} T(i,j,\tau)
\end{equation}
It follows immediately that
\begin{cor}
 $\mathrm{ch}_{\Vorb}(\tau)$ is modular invariant.
\end{cor}
\begin{proof}
  By Theorem \ref{thm:mod}, a modular transformation only permutes terms in the sum in equation (\ref{eq:orbchar}).
\end{proof}

In order to calculate characters it will be useful to define further modular invariants by splitting (\ref{eq:orbchar}) into orbits under the modular group.

\begin{thm}\label{thm:core}
 For $t\mid n$ define
 \begin{equation}\label{eq:core}
    C_t(\tau) = \sum_{i,j: (i,j,n) = t} T(i,j,\tau)\ . 
 \end{equation}
 Then $C_t$ is a modular function for $\Gamma$.
 \end{thm}
\begin{proof}
 For $i,j \in \Z_n$ and $M \in \mathrm{SL}_2(\Z)$ we have $(i,j,n) = ((i,j)M,n)$ so that modular transformations only permute terms in the sum in equation (\ref{eq:core}).
\end{proof}
Note that we can express the character of $\Vorb$ as
\[
  \text{ch}_{\Vorb}(\tau) = \frac{1}{n}\sum_{t\mid n}C_t(\tau)\ .
\]

Recall that according to Corollary \ref{cor:mod} for $t \mid n$, $T(0,t,\tau)$ is a modular function for $\Gamma_1(n/t)$.
Then $C_t(\tau)$ is the sum of $\Gamma_1(n/t)$-inequivalent modular images of $T(0,t,\tau)$.
As should be expected, we find that the number of terms in equation (\ref{eq:core}) is equal to the index of $\Gamma_1(n/t)$ in $\Gamma$:
\[
\lvert\{(i,j) \in \Z_{n/t}:(i,j,n/t) = 1\}\rvert = [\Gamma:\Gamma_1(n/t)] = \big(\frac{n}{t}\big)^2\prod_{p\mid n/t, p \text{ prime}}\bigg(1-\frac{1}{p^2}\bigg).
\]

In order to calculate the $C_t(\tau)$ it will be convenient to first introduce a further modular invariant $D_t(\tau)$
\begin{lem}\label{lem:D}
For $t \mid n$ define
 \begin{equation}\label{eq:D}
  D_t(\tau):= \sum_{(j,n)=t} T(0,j,\tau)\ . 
 \end{equation}
 Then $D_t$ is invariant under $\Gamma_0(n/t)$.
\end{lem}
\begin{proof}
	From Corollary~\ref{cor:mod} we know that $T(0,j,\tau)$ invariant under $\Gamma_1(n/t)$, with $t = (n,j)$.
	$T(0,t,\tau)$ is mapped to  $T(0,j,\tau)$, with $(n,j) = t$ by the representative $\begin{pmatrix} * & * \\ n/t & j/t \end{pmatrix}\in \Gamma_1(n/t)\backslash\Gamma_0(n/t)$.
	Furthermore $[\Gamma_0(n/t):\Gamma_1(n/t)] = |\{j \in \Z_n:(n,j) = t\}| = \varphi(n/t)$, where $\varphi$ is Euler's totient function,
	which follows immediately from standard result on congruence subgroups as presented, for example, in \cite{MR1474964}.
	Hence $\{T(0,j,\tau): (n,j) = t, j \in \Z_n\}$ is the set of inequivalent $\Gamma_0(n/t)$-images of $T(0,t,\tau)$ and their sum is modular function for $\Gamma_0(n/t)$. 
\end{proof}
It follows immediately that $C_t(\tau)$ is the sum of $\Gamma_0(n/t)$-inequivalent modular images of $D_t(\tau)$.
\begin{cor}\label{cor:Sp}
 It is easy to see that for a prime $p$ a set of representatives for $\Gamma_0(p)\backslash\Gamma$ is given by $\{\text{id}\}\cup\{ST^i:i = 0, \ldots, p-1\}$.
\end{cor}

\section{Lattice VOAs and their Automorphisms}
\subsection{Lattice VOAs}
The affine Lie algebra, called the \textit{Heisenberg current algebra}, associated with the complexified lattice $\mathfrak{h} = L\otimes_{\Z} \C$ is given by 
\[
 \hat{\mathfrak{h}} = \big(\mathfrak{h}\otimes_{\C} \C[t,t^{-1}]\big) \oplus \C \mathbf{k},
\]
with the Lie bracket defined by the linear continuation of
\[
 [x(n),y(m)] = \langle x,y \rangle n\delta_{n+m}\mathbf{k} \text{ and } [u,\mathbf{k}] = 0
\]
for $x,y \in \mathfrak{h}$,$n,m \in \Z$ and $u \in \hat{\mathfrak{h}}$ where we use the shorthand $x(n) := x\otimes t^n$.

The \textit{twisted group algebra} $\C_\epsilon[L]$ corresponding to the lattice $L$ is spanned by the $\C$-basis $\{\mathfrak{e}_\alpha\}_{\alpha \in L}$ and the multiplication is defined by
\[
 \mathfrak{e}_\alpha\mathfrak{e}_\beta = \epsilon(\alpha,\beta)\mathfrak{e}_{\alpha+\beta},
\]
where $\epsilon: L \times L \to \{\pm 1\}$ is a $2$-cocycle satisfying
\begin{equation}\label{eq:eps}
  \epsilon(\alpha,\alpha) = (-1)^{\langle\alpha,\alpha\rangle/2} \text{ and } \epsilon(\alpha,\beta)/\epsilon(\beta,\alpha) = (-1)^{\langle\alpha,\beta\rangle}
\end{equation}
for $\alpha,\beta \in L$. We define the weight by $\text{wt}(\mathfrak{e}_\alpha) = \langle\alpha,\alpha\rangle/2$.

The \textit{Lattice Vertex Operator Algebra} $V_L$ corresponding to the lattice $L$ is spanned by elements

\[
 h_k(-n_k)\ldots h_1(-n_1)1\otimes \mathfrak{e}_\alpha
\]
 with $n_k,\ldots, n_1 \geq 0$, where we let the Heisenberg current algebra act on the vector space $\C$ as
 \[
  h(n)\cdot w = 0 \text{ and } \mathbf{k}\cdot w = w,
 \]
for all $w \in \C$, $h \in \mathfrak{h}$ and $n \in \Z_{\geq 0}$. The weight of this element is given by
\[
 n_1 + \dots + n_k +\frac{1}{2}\langle\alpha,\alpha\rangle \in \Z_{\geq0}.
\]

\subsection{Automorphisms of Lattice Vertex Operator Algebras}

We can obtain an automorphism $\hat{\nu} = \nu_{\mathfrak{h}} \otimes \nu_{\epsilon}$ of the lattice vertex operator algebra $V_L$ by lifting an automorphism $\nu$ of $L$.
$\nu$ acts naturally on the Heisenberg current algebra to give $\nu_{\mathfrak{h}}$ via
\[
 \nu_{\mathfrak{h}}h_k(-n_k)\ldots h_1(-n_1)1 = (\nu h_k)(-n_k)\ldots (\nu h_1)(-n_1)1
\]

It can be shown that $\hat{\nu}$ is an automorphism of $V_L$ if and only if $\nu_{\epsilon}$ is an automorphism of the twisted group algebra $\C_\epsilon[L]$ such that
\[
 \nu_{\epsilon}(\mathfrak{e}_{\alpha}\mathfrak{e}_{\beta}) = \nu_{\epsilon}(\mathfrak{e}_{\alpha})\nu_{\epsilon}(\mathfrak{e}_{\beta}).
\]

Then $\nu_{\epsilon}$ satisfies
\begin{equation}\label{eq:epslift}
  \nu_{\epsilon}(\mathfrak{e}_{\alpha}) = u(\alpha)\mathfrak{e}_{\nu\alpha},
\end{equation}
where $u: L \to \C^*$ is a function satisfying
\begin{equation}\label{eq:u}
 \frac{\epsilon(\alpha,\beta)}{\epsilon(\nu\alpha,\nu\beta)} = \frac{u(\alpha)u(\beta)}{u(\alpha+\beta)}.
\end{equation}
As an immediate consequence we have the following result:
\begin{lem}\label{lem:hom}
 The restriction of $u$ to the fixed-point sublattice $L^{\nu}$ is a homomorphism of abelian groups.
\end{lem}

\begin{defn}\label{defn:ol}
 Note that $\frac{\epsilon(\alpha,\beta)}{\epsilon(\nu\alpha,\nu\beta)} \in \mathcal{B}^2(L,\{\pm1\})$ is a $2$-coboundary so that a function $u: L \to \{\pm1\}$ satisfying equation (\ref{eq:u}) always exists.
 The group of all lifted lattice automorphisms as above such that $u: L \to \{\pm1\}$ will be denoted by $O(\hat{L})$.
\end{defn}

\begin{thm}\label{thm:lift}
 The lift $\hat{\nu} = \nu_{\mathfrak{h}} \otimes \nu_{\epsilon}$ as defined above is an automorphism of the lattice vertex operator algebra $V_L$.
\end{thm}

\begin{defn}[Standard Lift]\label{defn:stdlift}
 Given an automorphism $\nu \in \text{Aut}(L)$ the function $u$ can always be chosen such that 
 \[
  u(\alpha) = 1 \text{ for all } \alpha \in L^{\nu},
 \]
such that $\nu_{\epsilon}(\mathfrak{e}_{\alpha}) = \mathfrak{e}_{\alpha}$ for all $\alpha \in L^{\nu}$. We call such a lift a \textit{standard lift}.
\end{defn}

If $\nu_{\epsilon}$ acts on $\mathfrak{e}_{\alpha}$ as in Equation \ref{eq:epslift} then $\nu_{\epsilon}^k$ acts as
\begin{equation}\label{eq:w}
 \nu_{\epsilon}^k(\mathfrak{e}_{\alpha}) = w_k(\alpha)\mathfrak{e}_{\nu^k\alpha},
\end{equation}
where $w_k(\alpha) = u(\alpha)u(\nu\alpha)\ldots u(\nu^{k-1}\alpha)$.

Note that $w_k(\alpha)$ satisfies Equation \ref{eq:u} for the automorphism $\nu^k \in \text{Aut}(L)$ so that $\nu_{\epsilon}^k$ is a lift of $\nu^k$.
The restriction of $w$ to the fixed point lattice $L^{\nu}$ is in particular given by
\be\label{eq:w_fp}
  w_k(\alpha) = u(\alpha)^k, \text{ for } \alpha \in L^{\nu}.
\ee
If $\nu_{\epsilon}$ is a standard lift then in general this will not be the case for $\nu_{\epsilon}^k$. In fact, we have the following result:

\begin{thm}\label{thm:powerlift}\cite{2016arXiv161109843M,MR1172696}
 Let $\nu_{\epsilon}$ be a lift of $\nu \in \text{Aut}(L)$ of order $m$ as in Equation \ref{eq:epslift}. Then for all $k \in \Z_{\geq0}$,
 \[
  \nu_{\epsilon}^k(\mathfrak{e}_{\alpha}) = (-1)^{\langle\alpha,\sum_{i = 0}^{k-1}\nu^i \alpha \rangle}u\bigg(\sum_{i = 0}^{k-1}\nu^i \alpha\bigg)\mathfrak{e}_{\alpha} = \mathfrak{e}_{\alpha}u\bigg(\sum_{i = 0}^{k-1}\nu^i \alpha\bigg) \cdot \begin{cases}
                                                                                                                                                                                             1 & \text{ if } m \text{ or } k \text{ is odd},\\
                                                                                                                                                                                             (-1)^{\langle\alpha,\nu^{k/2}\alpha\rangle} & \text{ if } m \text{ and } k \text{ are even},
                                                                                                                                                                                           \end{cases}
\]

for all $\alpha \in L^{\nu^k}$.

In other words, the restriction of $w_k(\alpha)$ to the fixed-point sublattice $L^{\nu^k}$ is given by
\be
 w_k(\alpha) = u\bigg(\sum_{i = 0}^{k-1}\nu^i \alpha\bigg) \cdot  \begin{cases}
                                                             1 & \text{ if } m \text{ or } k \text{ is odd},\\
                                                             (-1)^{\langle\alpha,\nu^{k/2}\alpha\rangle} & \text{ if } m \text{ and } k \text{ are even},
                                                       \end{cases}
\ee
for all $\alpha \in L^{\nu^k}$.
\end{thm}
\begin{rem}
 Note that $\sum_{i = 0}^{k-1}\nu^i \alpha \in L^{\nu}$ for all $\alpha \in L^{\nu^k}$. It follows that if $L^\nu = \{0\}$ then $\nu_{\epsilon}^k$ is a standard lift for all $k$.
\end{rem}

An immediate consequence of this is the following corollary on the order of lifted automorphisms:
\begin{cor}[Order of lifted automorphisms]\label{cor:ordlift}
 Let $\hat{\nu}$ be a standard lift of $\nu$ of order $m$. If $m$ is odd, then $\hat{\nu}$ has order $m$.
 If $m$ is even, then $\hat{\nu}$ has order $m$ if ${\langle\alpha,\nu^{m/2}\alpha\rangle} \in 2\Z$ for all $\alpha \in L$ and order $2m$ otherwise.
\end{cor}

In fact, \cite{MR1745258} give a complete description of all automorphisms of $V_L$, which we can use to construct non-standard lifts.
Let $O(\hat L)$ be as in Definition \ref{defn:ol} and 
define
\be
N = \langle e^{a_0} : a \in V_{(1)}\rangle.
\ee
\begin{thm}\cite{MR1745258}
 	Let $L$ be a positive definite even lattice. Then \
	$$ Aut(V_L) = N \cdot O(\hat L) $$
\end{thm}
In particular, if $L$ has no vectors of length $2$ we have
\[
 N = \langle e^{h_0} : h \in \mathfrak{h} \rangle.
\]
\begin{cor}\label{cor:beta}
 This implies in particular that all lifts $\nu_{\epsilon}$ of $\nu$ act as
 \be
  \nu_{\epsilon}(\mathfrak{e}_{\alpha}) = e^{2\pi i \langle \beta,\alpha\rangle}\tilde{u}(\alpha)\mathfrak{e}_{\nu\alpha},
 \ee
  for some vector $\beta \in \mathfrak{h}$, such that the restriction of $\tilde{u}$ to the fixed point lattice is unity.
\end{cor}
The automorphism $e^{h_0}$ is conjugate to $e^{\frac{1}{n}(\sum_{i=1}^n\nu^i h)_0}$ \cite{2017arXiv170400478V} hence we may choose $\beta \in \mathfrak{h}^{\nu}$ without loss of generality.

\subsection{Characters of Twisted Modules for Lattice Vertex Operator Algebras}
We can now use such an automorphism $\hat\nu$ to construct twisted modules and obtain their characters.
For lattice VOAs there is in principle an explicit expression for the twisted sector $V(g)$ and the action of $G$ on it,
so that one could obtain the trace directly. For cyclic orbifolds however we can avoid this. The idea is to use Theorem~\ref{thm:mod} to obtain all twisted characters from the untwisted characters $T(0,j,\tau)$ by applying $\mathrm{SL}_2(\Z)$ transformations.
More precisely, we can obtain the untwisted characters $T(0,j,\tau)$ directly, and use them to obtain the $D_t(\tau)$. From this we will then be able to recover $\mathrm{ch}_\Vorb(\tau)$ through modular transformations. The following results will be central to this:

\begin{thm}\label{thm:lattwtr}
 Let $\hat{\nu}$ be an automorphism of $V_L$ obtained as a lift of a lattice automorphism $\nu \in \text{Aut}(L)$ defined by a function $u: L \to \C^*$, where $\nu$ has cycle type $\prod_{t\mid n}t^{b_t}$.
 Then the twisted character for $\hat{\nu}$ on $V_L$  is given by
\be  \label{twtrace}
\text{tr}_{V_L}\hat{\nu}q^{L_0-c/24} = \frac{\vartheta_{L^{\nu},u}(\tau)}{\eta_{\nu}(\tau)},
\ee
where $\vartheta_{L^{\nu},u}(\tau)$ is the generalised theta function of the fixed-point sublattice $L^{\nu}$ given by
\[
 \vartheta_{L^{\nu},u}(\tau) = \sum_{\alpha \in L^{\nu}} u(\alpha)q^{\langle\alpha,\alpha\rangle/2}\ ,
\]
and the Eta-quotient $\eta_\nu(\tau)$ is given by
\be
\eta_\nu(\tau)=\prod_{t \mid n} \eta(t\tau)^{b_t}\ .
\ee
\end{thm}
This follows from a straightforward computation.
For a definition of the cycle type of $\nu$, see appendix~\ref{app:cycletype}. For the transformation properties of eta quotients, see appendix~\ref{app:etaquotient}.

For a standard lift (Definition \ref{defn:stdlift}) this clearly reduces to the familiar result
\[
\text{tr}_{V_L}\hat{\nu}q^{L_0-c/24} = \frac{\vartheta_{L^{\nu}}(\tau)}{\eta_{\nu}(\tau)},
\]
where $\vartheta_{L^{\nu}}(\tau)$ is the ordinary theta function of the fixed-point sublattice $L^{\nu}$.
We find the corresponding result for the automorphism $\hat{\nu}^k$:
 \[
  \text{tr}_{V_L}\hat{\nu}^k q^{L_0-c/24} = \frac{\vartheta_{L^{\nu^k},w}(\tau)}{\eta_{\nu^k}(\tau)},
 \]
with $w$ given by equation (\ref{eq:w}).
In the language of twisted trace functions for the cyclic automorphism group $\langle\hat{\nu}\rangle$ as defined in equation (\ref{eq:twtr}) this means
\begin{equation}\label{eq:lattwtr}
 T(0,j,\tau) = \text{tr}_{V_L}\hat{\nu}^j q^{L_0-c/24} = \frac{\vartheta_{L^{\nu^j},w_j}(\tau)}{\eta_{\nu^j}(\tau)},
\end{equation}
with $w_j$ defined as in equation (\ref{eq:w}) for the appropriate power.

We can now deduce the conformal weight of the unique irreducible $\hat\nu$-twisted $V_L$-module
\begin{thm}
	Let $\nu$ be a lattice automorphism of order $n$ with cycle type $\prod_{t\mid n}t^{b_t}$ and let $\hat\nu$ be a lift of $\nu$ as in Corollary \ref{cor:beta}.
	Then the the unique irreducible $\hat\nu$-twisted $V_L$-module has conformal weight
	\be
	\rho_{\hat{\nu}} = \frac{c}{24}-\frac{1}{24}\sum_{t\mid n}\frac{b_t}{t} + \frac{1}{2}\min(L'+\beta),
	\ee
	 where $\min(L'+\beta)$ is the squared length of a minimal element of $L' + \beta$.
\end{thm}
\begin{proof}
 Apply the S-transformation to the twisted trace $T(0,1,\tau)$ as defined in equation (\ref{eq:lattwtr}) and use Corollary \ref{cor:S_eta} and Theorem \ref{thm:inv}.
\end{proof}

We are now ready to give a general expression for $D_t$ as defined in Lemma \ref{lem:D} for lattice vertex operator algebras.
\begin{thm}\label{thm:Dt}
 \[
  D_t(\tau) = \frac{\sum_{d\mid \frac{n}{t}}\frac{n}{td} \mu\big(\frac{n}{td}\big) \vartheta_{K_{t,d}}(\tau)}{\eta_{\nu^t}(\tau)},
 \]
  where $K_{t,d}$ is the kernel of the restriction of $w_t^d$ to the fixed-point lattice $L^{\nu^t}$.
\end{thm}
\begin{proof}
 Using Equation \ref{eq:w_fp} and the M\"obius inversion formula we have
 \begin{align*}
  \eta_{\nu^t}(\tau)D_t(\tau) & = \sum_{\alpha \in L^{\nu^t}} \big(\sum_{(k,\frac{n}{t}) = 1} w_t^k(\alpha)\big) q^{\langle \alpha, \alpha \rangle / 2} \\
  & = \sum_{\alpha \in L^{\nu^t}} \bigg(\sum_{d \mid \frac{n}{t}} \mu\big(\frac{n}{td}\big) \big(\sum_{k=1}^{\frac{n}{td}} w_t^{kd}(\alpha) \big)\bigg)q^{\langle \alpha, \alpha \rangle / 2} \\
  & = \sum_{d \mid \frac{n}{t}} \mu\big(\frac{n}{td}\big)\bigg(\sum_{\alpha \in L^{\nu^t}} \big(\sum_{k=1}^{\frac{n}{td}} w_t^{kd}(\alpha) \big)q^{\langle \alpha, \alpha \rangle / 2}\bigg)
 \end{align*}
 Now $w_t^d$ is a $\frac{n}{td}$-th root of unity for all $t$ and $d$, hence $\sum_{k=1}^{\frac{n}{td}} w_t^{kd}(\alpha)$ is equal to $\frac{n}{td}$ if $\alpha$ is in the kernel and vanishes otherwise.
 The stated result follows.
\end{proof}
Furthermore, we can show that the modular transformation properties of the lattice theta functions in the above theorem are in fact related.
\begin{lem}
 The theta functions in the above theorem all transform with the same character.
\end{lem}
\begin{proof}
 Let $L$ be a lattice with basis $\{v_j\}$, $w:L \to \C^*$ a homomorphism such that $n$ is the least integer with $w(\alpha)^n = 1$ for all $\alpha \in L$
 and $w(v_j) = e^{2\pi i l_j / n}$ for some $l_j$ for all $j$.
 
 Then there exists a $j_0$ such that $(l_{j_0},n) = 1$ and for every $j \neq j_0$ we can find a $\kappa_j \in \Z$ such that $w(v_j + \kappa_j v_{j_0})= 1$.
 Then $\{v_{j_0}\} \cup \{v_j + \kappa_j v_{j_0}:j \neq j_0\}$ is a basis of $L$ and $\{n v_{j_0}\} \cup \{v_j + \kappa_j v_{j_0}:j \neq j_0\}$ is a basis of $\text{ker}(w)$.
 It follows that $\text{det}(\text{ker}(w)) = n^2\text{det}(L)$ and hence that the theta functions of $L$ and $\text{ker}(w)$ transform under the same character.
\end{proof}

\begin{note}
 Note that the $K_{t,d}$ will be a full rank sublattice of $L^{\nu^t}$. This is particularly problematic when $t$ is such that $L^{\nu^t} = L$, as 
 if the rank of $L^{\nu^t}$ and therefore of $K_{t,d}$ is large the computational cost of calculating the theta function to sufficiently high order may be prohibitive.
 This is the limiting factor in our computations.
\end{note}

\section{Orbifolds of extremal lattices}
We are interested in VOAs which come from even self-dual extremal lattices in $d=48$ and $d=72$. For $d=48$, four such lattices are known \cite{MR1662447,MR1489922,MR3225314}, and in $d=72$ one \cite{MR2999133}. 
Their information is listed in table~\ref{t:lattices}.
As mentioned before, we are interested in extremal lattices because we want to construct VOAs with few low weight states.
To do this as systematically as possible, we use the following approach. For a given lattice, we use MAGMA to first identify all 
conjugacy classes of cyclic groups of $\mathrm{Aut}(L)$ and their generators $g$. For each generator we then proceed on a case by case basis.

\begin{enumerate}
	\item In the simplest case, $g$ and all its powers have no fixed point lattices. This turns out to be a fairly common case. 
	 In that case the only lift is the standard lift, and $\hat g$ has the same order as $g$.
	The $T(0,j,\tau)$ are simply eta-quotients, whose $\mathrm{SL}_2(\Z)$ transformation properties we know from proposition~\ref{prop:etatrans},
	so that we can obtain all $T(i,j,\tau)$ from Theorem \ref{thm:mod}.
	We can compute the type of all orbifolds, and only keep the ones of type 0 to construct holomorphic orbifolds VOA $V^{orb(\langle \hat g \rangle)}$.
\setcounter{saveenumi}{\theenumi}
\end{enumerate}	
	
	If $g$ has a non-trivial fixed point lattice $L^g$, there are more options. We can still use a standard lift to obtain $\hat g$,
	but in this case it can happen by Corollary~\ref{cor:ordlift} that the order of $\hat g$ is double the order of $g$. Again we are looking for type 0 orbifolds.
	If the standard lift $\hat g$ does not have type 0, we can try to use a non-standard lift of $g$ to obtain a VOA automorphism which does have type 0.
	In the cases at hand we could always find such a non-standard lift. We will discuss no-standard lifts below, and first discuss the case where $\hat g$ and all its powers are standard lifts.
	
\begin{enumerate}
\setcounter{enumi}{\thesaveenumi}
	\item
	If the order of $\hat g$ is prime, and all powers of $\hat g$ are standard lifts, then all the $T(0,j,\tau)$ are equal to a product of an eta quotient and a (ordinary) lattice theta function
	and by Corollary \ref{cor:Sp} all the remaining twisted characters can be obtained from the $T(i,0,\tau)$ by applying $T$-transformations.
	In order to calculate $T(i,0,\tau)$ we express the $S$-transformation of the lattice theta function  $\vartheta_{L^g}(\tau)$ in terms of the theta function of the dual lattice $(L^g)'$
	using the inversion formula \ref{thm:inv} in appendix~\ref{app:inverse} 
	$\vartheta_{L^g}(-1/\tau)= (\det L^g)^{-\frac{1}{2}}(-i\tau)^{\mathrm{Rank}(L^g)} \vartheta_{(L^g)'}(\tau)$.
	Subsequent summation over $T$-images will only remove non-integer orders in the $q$-expansion of $T(i,0,\tau)$. See also \cite{1507.08142}.
\item If the order of $\hat g$ is not prime, but all powers of $\hat g$ are standard lifts, then all the $T(0,j,\tau)$ are equal to a product of an eta quotient and a (ordinary) lattice theta function.
The theorem of Hecke-Schoenberg tells us that those lattice theta functions are modular forms of $\Gamma_0(N)$ for some level $N$ of some weight $k$, 
possibly with some character $\chi$. (Alternatively we could also apply Lemma~\ref{lem:D} to establish this for $D_t(\tau)$.) 
We can try to express these in terms of eta quotients.
In all cases at hand, following the approach of Rouse and Webb \cite{RW13} we are able to find a basis of $\mathcal{M}_k(\Gamma_0(n))$ in terms of eta-quotients by virtue of Theorem~\ref{thm:quomod}
--- see table~\ref{t:etabasis} for these bases.
This allows us to express $D_t(\tau)$ as sums of eta quotients, so that we can read off the $\mathrm{SL}_2(\Z)$-transformations.
\setcounter{saveenumi}{\theenumi}
\end{enumerate}	
The resulting characters of orbifold VOAs constructed in (1) through (3) are listed under `Standard lift without order doubling' in the tables below. 
They form indeed the majority of the cases we analysed. Next let us discuss non-standard lifts. 
\begin{enumerate}
	\setcounter{enumi}{\thesaveenumi}
	\item
In the simplest case we still take $\hat g$ to be a standard lift, but have some $\hat g^k$ that are not standard, such as in the case of order doubling described in Corollary~\ref{cor:ordlift}. $\hat g^k$ thus leads to generalized theta functions with phases. We can however use Theorem~\ref{thm:Dt} to express those in terms of standard theta functions, and use Hecke-Schoenberg again just as before. We listed these cases under `Standard lift with order doubling' in the tables below.
\item Finally, we can consider cases where the standard lift for $g$ does not give an orbifold of type 0.
In this case we can consider instead more general lifts of the form of Corollary~\ref{cor:beta}.
The idea is to pick a vector $\beta$, which increases the order $\hat g$ so that the resulting orbifold becomes type 0.
We were able to find at least one such $\beta$ for every orbifold with a non-vanishing fixed point lattice.
Again we get generalized theta functions with phases, and use Theorem~\ref{thm:Dt} to rewrite them in terms of standard theta functions.
We listed these cases under `Non-standard lift' in the tables below.
\end{enumerate}

In the following we list all holomorphic extensions that we could find. For the lattices in $d = 48$, this covers all cyclic orbifolds with vanishing fixed point lattice
and gives one example for every cyclic orbifold with non-vanishing fixed point lattice without order doubling. We did however not systematically construct all possible lifts in those cases.
For the lattice $\Gamma_{72}$ we list orbifolds for all cyclic groups such every element is a standard lift. 
Our constructions in particular give the VOAs listed in theorem~\ref{thmmain}.

\renewcommand{\arraystretch}{1.2}

\begin{table}\label{t:lattices}
\begin{tabular}{|ccc|}
\hline
\toprule
$L$ & $\text{Aut}(L)$ & $|\text{Aut}(L)|$\\
\midrule
$P_{48m}$&${}$& multiple of $1200=2^4~3~5^2$\\
$P_{48n}$&$(\mathrm{SL}_2(13) \mathsf{Y} \mathrm{SL}_2(5))\cdot 2^2$&$524160 = 2^7~3^2~5^7~13$\\
$P_{48p}$&$(\mathrm{SL}_2(23)\times S_3) : 2$&$72864 = 2^5~3^2~11~23$\\
$P_{48q}$&$\mathrm{SL}_2(47)$&$103776 = 2^5~3~23~47$\\
$\Gamma_{72}$&$(\mathrm{SL}_2(25)\times \mathrm{PSL}_2(7)) : 2$&$241600 = 2^8~3^2~5^2~7~13$\\
\hline
\end{tabular}
\caption{Known extremal lattices in $d=48$ and $72$. Taken from \cite{MR3225314}. Explicit expressions for the Gram matrix and the generators of the automorphism groups were taken from \cite{latticewebsite}}
\end{table}

\subsection{Cyclic orbifolds for the Lattice $P_{48m}$}

\begin{center}
\begin{longtable}{llll}
\toprule
$n$ & $C_g$     & $\text{Rank}(L^g)$ & $\text{ch}_{V^{orb(\langle \tilde g \rangle)}}(q)$ \\ \midrule \endhead
\multicolumn{4}{c}{Standard lift without order doubling}\\
\midrule
$1$ & $1^{48}$ & $48$ & $q^{-2} + 48q^{-1} + 1224 + \mathcal{O}(q)$ \\
$2$ & $1^{-48}2^{48}$ & $0$ & $q^{-2} + 1176 + \mathcal{O}(q)$ \\
$3$ & $1^{-24}3^{24}$ & $0$ & $q^{-2} + 576 + \mathcal{O}(q)$ \\
$4$ & $2^{-24}4^{24}$ & $0$ & $q^{-2} + 576 + \mathcal{O}(q)$ \\
$4$ & $2^{-24}4^{24}$ & $0$ & $q^{-2} + 576 + \mathcal{O}(q)$ \\
$5$ & $1^{-2}5^{10}$ & $8$ & $q^{-2} + 8q^{-1} + 264 + \mathcal{O}(q)$ \\
$5$ & $1^{-2}5^{10}$ & $8$ & $q^{-2} + 8q^{-1} + 264 + \mathcal{O}(q)$ \\
$5$ & $1^{-12}5^{12}$ & $0$ & $q^{-2} + 288 + \mathcal{O}(q)$ \\
$5$ & $1^{8}5^{8}$ & $16$ & $q^{-2} + 16q^{-1} + 456 + \mathcal{O}(q)$ \\
$6$ & $1^{24}2^{-24}3^{-24}6^{24}$ & $0$ & $q^{-2} + 1176 + \mathcal{O}(q)$ \\
$10$ & $1^{2}2^{-2}5^{-10}10^{10}$ & $0$ & $q^{-2} + 312 + \mathcal{O}(q)$ \\
$10$ & $1^{2}2^{-2}5^{-10}10^{10}$ & $0$ & $q^{-2} + 312 + \mathcal{O}(q)$ \\
$10$ & $1^{-8}2^{8}5^{-8}10^{8}$ & $0$ & $q^{-2} + 408 + \mathcal{O}(q)$ \\
$10$ & $1^{12}2^{-12}5^{-12}10^{12}$ & $0$ & $q^{-2} + 600 + \mathcal{O}(q)$ \\
$12$ & $2^{12}4^{-12}6^{-12}12^{12}$ & $0$ & $q^{-2} + 648 + \mathcal{O}(q)$ \\
$15$ & $1^{1}3^{-1}5^{-5}15^{5}$ & $0$ & $q^{-2} + 192 + \mathcal{O}(q)$ \\
$15$ & $1^{6}3^{-6}5^{-6}15^{6}$ & $0$ & $q^{-2} + 360 + \mathcal{O}(q)$ \\
$15$ & $1^{1}3^{-1}5^{-5}15^{5}$ & $0$ & $q^{-2} + 192 + \mathcal{O}(q)$ \\
$15$ & $1^{-4}3^{4}5^{-4}15^{4}$ & $0$ & $q^{-2} + 192 + \mathcal{O}(q)$ \\
$20$ & $2^{6}4^{-6}10^{-6}20^{6}$ & $0$ & $q^{-2} + 360 + \mathcal{O}(q)$ \\
$30$ & $1^{-6}2^{6}3^{6}5^{6}6^{-6}10^{-6}15^{-6}30^{6}$ & $0$ & $q^{-2} + 600 + \mathcal{O}(q)$ \\
$30$ & $1^{4}2^{-4}3^{-4}5^{4}6^{4}10^{-4}15^{-4}30^{4}$ & $0$ & $q^{-2} + 408 + \mathcal{O}(q)$ \\
$30$ & $1^{-1}2^{1}3^{1}5^{5}6^{-1}10^{-5}15^{-5}30^{5}$ & $0$ & $q^{-2} + 312 + \mathcal{O}(q)$ \\
$30$ & $1^{-1}2^{1}3^{1}5^{5}6^{-1}10^{-5}15^{-5}30^{5}$ & $0$ & $q^{-2} + 312 + \mathcal{O}(q)$ \\ 
\midrule
\multicolumn{4}{c}{Standard lift with order doubling}\\
\midrule
$4$ & $2^{24}$ & $24$ & $q^{-2} + 24q^{-1} + 1896 + \mathcal{O}(q)$ \\
$20$ & $2^{4}10^{4}$ & $8$ & $q^{-2} + 8q^{-1} + 744 + \mathcal{O}(q)$ \\ 
\bottomrule
\end{longtable} 
\end{center}

\subsection{Cyclic orbifolds for the Lattice $P_{48n}$}

\begin{center}
\begin{longtable}{llll}
\toprule
$n$ & $C_g$     & $\text{Rank}(L^g)$ & $\text{ch}_{V^{orb(\langle \tilde g \rangle)}}(q)$ \\ \midrule \endhead
\multicolumn{4}{c}{Standard lift without order doubling}\\
\midrule
$1$ & $1^{48}$ & $48$ & $q^{-2} + 48q^{-1} + 1224 + \mathcal{O}(q)$ \\
$2$ & $2^{24}$ & $24$ & $q^{-2} + 24q^{-1} + 648 + \mathcal{O}(q)$ \\
$2$ & $1^{-48}2^{48}$ & $0$ & $q^{-2} + 1176 + \mathcal{O}(q)$ \\
$2$ & $2^{24}$ & $24$ & $q^{-2} + 24q^{-1} + 648 + \mathcal{O}(q)$ \\
$3$ & $1^{-24}3^{24}$ & $0$ & $q^{-2} + 576 + \mathcal{O}(q)$ \\
$4$ & $2^{-24}4^{24}$ & $0$ & $q^{-2} + 576 + \mathcal{O}(q)$ \\
$4$ & $2^{-24}4^{24}$ & $0$ & $q^{-2} + 576 + \mathcal{O}(q)$ \\
$4$ & $2^{-24}4^{24}$ & $0$ & $q^{-2} + 576 + \mathcal{O}(q)$ \\
$4$ & $2^{-24}4^{24}$ & $0$ & $q^{-2} + 576 + \mathcal{O}(q)$ \\
$4$ & $2^{-24}4^{24}$ & $0$ & $q^{-2} + 576 + \mathcal{O}(q)$ \\
$5$ & $1^{-12}5^{12}$ & $0$ & $q^{-2} + 288 + \mathcal{O}(q)$ \\
$6$ & $2^{-12}6^{12}$ & $0$ & $q^{-2} + 288 + \mathcal{O}(q)$ \\
$6$ & $2^{-12}6^{12}$ & $0$ & $q^{-2} + 288 + \mathcal{O}(q)$ \\
$6$ & $1^{24}2^{-24}3^{-24}6^{24}$ & $0$ & $q^{-2} + 1176 + \mathcal{O}(q)$ \\
$7$ & $1^{-8}7^{8}$ & $0$ & $q^{-2} + 192 + \mathcal{O}(q)$ \\
$10$ & $1^{12}2^{-12}5^{-12}10^{12}$ & $0$ & $q^{-2} + 600 + \mathcal{O}(q)$ \\
$10$ & $2^{-6}10^{6}$ & $0$ & $q^{-2} + 142 + \mathcal{O}(q)$ \\
$12$ & $2^{12}4^{-12}6^{-12}12^{12}$ & $0$ & $q^{-2} + 648 + \mathcal{O}(q)$ \\
$12$ & $2^{12}4^{-12}6^{-12}12^{12}$ & $0$ & $q^{-2} + 648 + \mathcal{O}(q)$ \\
$12$ & $2^{12}4^{-12}6^{-12}12^{12}$ & $0$ & $q^{-2} + 648 + \mathcal{O}(q)$ \\
$13$ & $1^{-4}13^{4}$ & $0$ & $q^{-2} + 96 + \mathcal{O}(q)$ \\
$14$ & $2^{-4}14^{4}$ & $0$ & $q^{-2} + 96 + \mathcal{O}(q)$ \\
$14$ & $1^{8}2^{-8}7^{-8}14^{8}$ & $0$ & $q^{-2} + 408 + \mathcal{O}(q)$ \\
$14$ & $2^{-4}14^{4}$ & $0$ & $q^{-2} + 96 + \mathcal{O}(q)$ \\
$20$ & $2^{6}4^{-6}10^{-6}20^{6}$ & $0$ & $q^{-2} + 360 + \mathcal{O}(q)$ \\
$21$ & $1^{4}3^{-4}7^{-4}21^{4}$ & $0$ & $q^{-2} + 246 + \mathcal{O}(q)$ \\
$26$ & $2^{-2}26^{2}$ & $0$ & $q^{-2} + 48 + \mathcal{O}(q)$ \\
$26$ & $1^{4}2^{-4}13^{-4}26^{4}$ & $0$ & $q^{-2} + 216 + \mathcal{O}(q)$ \\
$28$ & $2^{4}4^{-4}14^{-4}28^{4}$ & $0$ & $q^{-2} + 264 + \mathcal{O}(q)$ \\
$28$ & $2^{4}4^{-4}14^{-4}28^{4}$ & $0$ & $q^{-2} + 264 + \mathcal{O}(q)$ \\
$28$ & $2^{4}4^{-4}14^{-4}28^{4}$ & $0$ & $q^{-2} + 264 + \mathcal{O}(q)$ \\
$35$ & $1^{2}5^{-2}7^{-2}35^{2}$ & $0$ & $q^{-2} + 168 + \mathcal{O}(q)$ \\
$39$ & $1^{2}3^{-2}13^{-2}39^{2}$ & $0$ & $q^{-2} + 168 + \mathcal{O}(q)$ \\
$42$ & $2^{2}6^{-2}14^{-2}42^{2}$ & $0$ & $q^{-2} + 168 + \mathcal{O}(q)$ \\
$42$ & $2^{2}6^{-2}14^{-2}42^{2}$ & $0$ & $q^{-2} + 168 + \mathcal{O}(q)$ \\
$42$ & $1^{-4}2^{4}3^{4}6^{-4}7^{4}14^{-4}21^{-4}42^{4}$ & $0$ & $q^{-2} + 408 + \mathcal{O}(q)$ \\
$52$ & $2^{2}4^{-2}26^{-2}52^{2}$ & $0$ & $q^{-2} + 168 + \mathcal{O}(q)$ \\
$65$ & $1^{1}5^{-1}13^{-1}65^{1}$ & $0$ & $q^{-2} + 120 + \mathcal{O}(q)$ \\
$70$ & $1^{-2}2^{2}5^{2}7^{2}10^{-2}14^{-2}35^{-2}70^{2}$ & $0$ & $q^{-2} + 216 + \mathcal{O}(q)$ \\
$70$ & $2^{1}10^{-1}14^{-1}70^{1}$ & $0$ & $q^{-2} + 120 + \mathcal{O}(q)$ \\
$78$ & $2^{1}6^{-1}26^{-1}78^{1}$ & $0$ & $q^{-2} + 120 + \mathcal{O}(q)$ \\
$78$ & $1^{-2}2^{2}3^{2}6^{-2}13^{2}26^{-2}39^{-2}78^{2}$ & $0$ & $q^{-2} + 216 + \mathcal{O}(q)$ \\
$84$ & $2^{-2}4^{2}6^{2}12^{-2}14^{2}28^{-2}42^{2}84^{2}$ & $0$ & $q^{-2} + 192 + \mathcal{O}(q)$ \\
$130$ & $1^{-1}2^{1}5^{1}10^{-1}13^{1}26^{-1}65^{-1}130^{1}$ & $0$ & $q^{-2} + 120 + \mathcal{O}(q)$ \\ 
\midrule
\multicolumn{4}{c}{Standard lift with order doubling}\\
\midrule
$4$ & $2^{24}$ & $24$ & $q^{-2} + 24q^{-1} + 1896 + \mathcal{O}(q)$ \\
$8$ & $4^{12}$ & $12$ & $q^{-2} + 12q^{-1} + 936 + \mathcal{O}(q)$ \\ 
\midrule
\multicolumn{4}{c}{Non-standard lift}\\
\midrule
$9$ & $3^{16}$ & $16$ & $q^{-2} + 18q^{-1} + 1488 + \mathcal{O}(q)$\\
$9$ & $3^{16}$ & $16$ & $q^{-2} + 18q^{-1} + 1560 + \mathcal{O}(q)$\\
$18$ & $6^{8}$ & $8$ & $q^{-2} + 8q^{-1} + 816 + \mathcal{O}(q)$\\
$18$ & $6^{8}$ & $8$ & $q^{-2} + 8q^{-1} + 888 + \mathcal{O}(q)$\\ 
\bottomrule
\end{longtable} 
\end{center}

\subsection{Cyclic orbifolds for the Lattice $P_{48p}$}

\begin{center}
\begin{longtable}{llll}
\toprule
$n$ & $C_g$     & $\text{Rank}(L^g)$ & $\text{ch}_{V^{orb(\langle \tilde g \rangle)}}(q)$ \\ \midrule \endhead
\multicolumn{4}{c}{Standard lift without order doubling}\\
\midrule
$1$ & $1^{48}$ & $48$ & $q^{-2} + 48q^{-1} + 1224 + \mathcal{O}(q)$ \\
$2$ & $1^{-48}2^{48}$ & $0$ & $q^{-2} + 1176 + \mathcal{O}(q)$ \\
$3$ & $1^{-24}3^{24}$ & $0$ & $q^{-2} + 576 + \mathcal{O}(q)$ \\
$4$ & $2^{-24}4^{24}$ & $0$ & $q^{-2} + 576 + \mathcal{O}(q)$ \\
$4$ & $2^{-24}4^{24}$ & $0$ & $q^{-2} + 576 + \mathcal{O}(q)$ \\
$4$ & $2^{-24}4^{24}$ & $0$ & $q^{-2} + 576 + \mathcal{O}(q)$ \\
$6$ & $1^{24}2^{-24}3^{-24}6^{24}$ & $0$ & $q^{-2} + 1176 + \mathcal{O}(q)$ \\
$11$ & $1^{4}11^{4}$ & $8$ & $q^{-2} + 8q^{-1} + 264 + \mathcal{O}(q)$ \\
$12$ & $2^{12}4^{-12}6^{-12}12^{12}$ & $0$ & $q^{-2} + 648 + \mathcal{O}(q)$ \\
$12$ & $2^{12}4^{-12}6^{-12}12^{12}$ & $0$ & $q^{-2} + 648 + \mathcal{O}(q)$ \\
$22$ & $1^{-4}2^{4}11^{-4}22^{4}$ & $0$ & $q^{-2} + 216 + \mathcal{O}(q)$ \\
$23$ & $1^{2}23^{2}$ & $4$ & $q^{-2} + 4q^{-1} + 168 + \mathcal{O}(q)$ \\
$33$ & $1^{-2}3^{2}11^{-2}33^{2}$ & $0$ & $q^{-2} + 96 + \mathcal{O}(q)$ \\
$44$ & $2^{-2}4^{2}22^{-2}44^{2}$ & $0$ & $q^{-2} + 96 + \mathcal{O}(q)$ \\
$46$ & $1^{-2}2^{2}23^{-2}46^{2}$ & $0$ & $q^{-2} + 120 + \mathcal{O}(q)$ \\
$66$ & $1^{2}2^{-2}3^{-2}6^{2}11^{2}22^{-2}33^{-2}66^{2}$ & $0$ & $q^{-2} + 216 + \mathcal{O}(q)$ \\
$69$ & $1^{-1}3^{1}23^{-1}69^{1}$ & $0$ & $q^{-2} + 48 + \mathcal{O}(q)$ \\
$132$ & $2^{1}4^{-1}6^{-1}12^{1}22^{1}44^{-1}66^{-1}132^{1}$ & $0$ & $q^{-2} + 168 + \mathcal{O}(q)$ \\
$138$ & $1^{1}2^{-1}3^{-1}6^{1}23^{1}46^{-1}69^{-1}138^{1}$ & $0$ & $q^{-2} + 120 + \mathcal{O}(q)$ \\ 
\midrule
\multicolumn{4}{c}{Standard lift with order doubling}\\
\midrule
$4$ & $2^{24}$ & $24$ & $q^{-2} + 24q^{-1} + 744 + \mathcal{O}(q)$ \\
$4$ & $2^{24}$ & $24$ & $q^{-2} + 24q^{-1} + 1800 + \mathcal{O}(q)$ \\
$44$ & $2^{2}22^{2}$ & $4$ & $q^{-2} + 4q^{-1} + 264 + \mathcal{O}(q)$ \\
$44$ & $2^{2}22^{2}$ & $4$ & $q^{-2} + 4q^{-1} + 360 + \mathcal{O}(q)$ \\
$92$ & $2^{1}46^{1}$ & $2$ & $q^{-2} + 2q^{-1} + 216 + \mathcal{O}(q)$ \\
\midrule
\multicolumn{4}{c}{Non-standard lift}\\
\midrule
		$9$ &  $3^{16}$ & $16$ & $q^{-2} + 18q^{-1} + 1488 + \mathcal{O}(q)$ \\
$9$ &  $3^{16}$ & $16$ & $q^{-2} + 18q^{-1} + 1632 + \mathcal{O}(q)$ \\  \bottomrule
\end{longtable} 
\end{center}

\subsection{Cyclic orbifolds for the Lattice $P_{48q}$}

\begin{center}
	\begin{center}
  \begin{longtable}{llll}
\toprule
$n$ & $C_g$     & $\text{Rank}(L^g)$ & $\text{ch}_{V^{orb(\langle \tilde g \rangle)}}(q)$ \\ \midrule \endhead
\multicolumn{4}{c}{Standard lift without order doubling}\\
\midrule
$1$ & $1^{48}$ & $48$ & $q^{-2} + 48q^{-1} + 1224 + \mathcal{O}(q)$ \\
$2$ & $1^{-48}2^{48}$ & $0$ & $q^{-2} + 1176 + \mathcal{O}(q)$ \\
$4$ & $2^{-24}4^{24}$ & $0$ & $q^{-2} + 576 + \mathcal{O}(q)$ \\
$23$ & $1^{2}23^{2}$ & $4$ & $q^{-2} + 4q^{-1} + 168 + \mathcal{O}(q)$ \\
$46$ & $1^{-2}2^{2}23^{-2}46^{2}$ & $0$ & $q^{-2} + 120 + \mathcal{O}(q)$ \\
$47$ & $1^{1}47^{1}$ & $2$ & $q^{-2} + 2q^{-1} + 120 + \mathcal{O}(q)$ \\
$94$ & $1^{-1}2^{1}47^{-1}94^{1}$ & $0$ & $q^{-2} + 72 + \mathcal{O}(q)$ \\
\midrule
\multicolumn{4}{c}{Non-standard lift}\\
\midrule
		$9$ & $3^{16}$ & $16$ & $q^{-2} + 18q^{-1} + 1560 + \mathcal{O}(q)$\\
\bottomrule
\end{longtable} 
\end{center} 
\end{center}

\subsection{Cyclic orbifolds for the Lattice $\Gamma_{72}$}

\begin{center}
\begin{longtable}{llll}
\toprule
$n$ & $C_g$     & $\text{Rank}(L^g)$ & $\text{ch}_{V^{orb(\langle \tilde g \rangle)}}(q)$ \\ \midrule \endhead
\multicolumn{4}{c}{Standard lift without order doubling}\\
\midrule
$1$ & $1^{72}$ & $72$ & $q^{-3} + 72q^{-2} + 2700q^{-1} + 70080 + \mathcal{O}(q)$ \\
$2$ & $1^{-24}2^{48}$ & $24$ & $q^{-3} + 24q^{-2} + 1500q^{-1} + 37824 + \mathcal{O}(q)$ \\
$2$ & $1^{-72}2^{72}$ & $0$ & $q^{-3} + 2628q^{-1} + 5184 + \mathcal{O}(q)$ \\
$2$ & $1^{24}2^{24}$ & $48$ & $q^{-3} + 48q^{-2} + 1548q^{-1} + 40704 + \mathcal{O}(q)$ \\
$3$ & $3^{24}$ & $24$ & $q^{-3} + 24q^{-2} + 900q^{-1} + 23424 + \mathcal{O}(q)$ \\
$3$ & $3^{24}$ & $24$ & $q^{-3} + 24q^{-2} + 900q^{-1} + 23424 + \mathcal{O}(q)$ \\
$3$ & $3^{24}$ & $24$ & $q^{-3} + 24q^{-2} + 900q^{-1} + 23424 + \mathcal{O}(q)$ \\
$4$ & $1^{-24}4^{24}$ & $0$ & $q^{-3} + 876q^{-1} + 16128 + \mathcal{O}(q)$ \\
$4$ & $1^{24}2^{-24}4^{24}$ & $24$ & $q^{-3} + 24q^{-2} + 900q^{-1} + 23424 + \mathcal{O}(q)$ \\
$5$ & $1^{12}5^{12}$ & $24$ & $q^{-3} + 24q^{-2} + 612q^{-1} + 16512 + \mathcal{O}(q)$ \\
$5$ & $1^{-18}5^{18}$ & $0$ & $q^{-3} + 648q^{-1} + 13608 + \mathcal{O}(q)$ \\
$6$ & $3^{-24}6^{24}$ & $0$ & $q^{-3} + 876q^{-1} + 1728 + \mathcal{O}(q)$ \\
$6$ & $3^{-24}6^{24}$ & $0$ & $q^{-3} + 876q^{-1} + 1728 + \mathcal{O}(q)$ \\
$6$ & $3^{-8}6^{16}$ & $8$ & $q^{-3} + 8q^{-2} + 500q^{-1} + 12672 + \mathcal{O}(q)$ \\
$6$ & $3^{8}6^{8}$ & $16$ & $q^{-3} + 16q^{-2} + 516q^{-1} + 13632 + \mathcal{O}(q)$ \\
$6$ & $3^{-24}6^{24}$ & $0$ & $q^{-3} + 876q^{-1} + 1728 + \mathcal{O}(q)$ \\
$7$ & $1^{-12}7^{12}$ & $0$ & $q^{-3} + 432q^{-1} + 9936 + \mathcal{O}(q)$ \\
$10$ & $1^{18}2^{-18}5^{-18}10^{18}$ & $0$ & $q^{-3} + 648q^{-1} + 5928 + \mathcal{O}(q)$ \\
$10$ & $1^{-12}2^{12}5^{-12}10^{12}$ & $0$ & $q^{-3} + 588q^{-1} + 1728 + \mathcal{O}(q)$ \\
$10$ & $1^{4}2^{4}5^{4}10^{4}$ & $16$ & $q^{-3} + 16q^{-2} + 356q^{-1} + 9792 + \mathcal{O}(q)$ \\
$10$ & $1^{-6}2^{-6}5^{6}10^{6}$ & $0$ & $q^{-3} + 360q^{-1} + 7992 + \mathcal{O}(q)$ \\
$10$ & $1^{-4}2^{8}5^{-4}10^{8}$ & $8$ & $q^{-3} + 8q^{-2} + 340q^{-1} + 8832 + \mathcal{O}(q)$ \\
$10$ & $1^{6}2^{-12}5^{-6}10^{12}$ & $0$ & $q^{-3} + 360q^{-1} + 7872 + \mathcal{O}(q)$ \\
$12$ & $3^{8}6^{-8}12^{8}$ & $8$ & $q^{-3} + 8q^{-2} + 300q^{-1} + 7872 + \mathcal{O}(q)$ \\
$12$ & $3^{-8}12^{8}$ & $0$ & $q^{-3} + 292q^{-1} + 5376 + \mathcal{O}(q)$ \\
$13$ & $1^{-6}13^{6}$ & $0$ & $q^{-3} + 216q^{-1} + 5400 + \mathcal{O}(q)$ \\
$14$ & $1^{12}2^{-12}7^{-12}14^{12}$ & $0$ & $q^{-3} + 432q^{-1} + 3120 + \mathcal{O}(q)$ \\
$15$ & $3^{4}15^{4}$ & $8$ & $q^{-3} + 8q^{-2} + 204q^{-1} + 5568 + \mathcal{O}(q)$ \\
$15$ & $3^{-6}15^{6}$ & $0$ & $q^{-3} + 216q^{-1} + 4536 + \mathcal{O}(q)$ \\
$20$ & $1^{4}2^{-4}4^{4}5^{4}10^{-4}20^{4}$ & $8$ & $q^{-3} + 8q^{-2} + 204q^{-1} + 5568 + \mathcal{O}(q)$ \\
$20$ & $1^{6}4^{-6}5^{-6}20^{6}$ & $0$ & $q^{-3} + 216q^{-1} + 4200 + \mathcal{O}(q)$ \\
$20$ & $1^{-6}2^{6}4^{-6}5^{6}10^{-6}20^{6}$ & $0$ & $q^{-3} + 216q^{-1} + 4536 + \mathcal{O}(q)$ \\
$20$ & $1^{-4}4^{4}5^{-4}20^{4}$ & $0$ & $q^{-3} + 196q^{-1} + 3840 + \mathcal{O}(q)$ \\
$21$ & $3^{-4}21^{4}$ & $0$ & $q^{-3} + 144q^{-1} + 3312 + \mathcal{O}(q)$ \\
$26$ & $1^{6}2^{-6}13^{-6}26^{6}$ & $0$ & $q^{-3} + 216q^{-1} + 1176 + \mathcal{O}(q)$ \\
$26$ & $1^{-2}2^{-2}13^{2}26^{2}$ & $0$ & $q^{-3} + 120q^{-1} + 3144 + \mathcal{O}(q)$ \\
$26$ & $1^{2}2^{-4}13^{-2}26^{4}$ & $0$ & $q^{-3} + 120q^{-1} + 3072 + \mathcal{O}(q)$ \\
$30$ & $3^{-4}6^{4}15^{-4}30^{4}$ & $0$ & $q^{-3} + 196q^{-1} + 576 + \mathcal{O}(q)$ \\
$30$ & $3^{6}6^{-6}15^{-6}30^{6}$ & $0$ & $q^{-3} + 216q^{-1} + 2040 + \mathcal{O}(q)$ \\
$35$ & $1^{-2}5^{-2}7^{2}35^{2}$ & $0$ & $q^{-3} + 96q^{-1} + 2352 + \mathcal{O}(q)$ \\
$35$ & $1^{3}5^{-3}7^{-3}35^{3}$ & $0$ & $q^{-3} + 108q^{-1} + 2472 + \mathcal{O}(q)$ \\
$39$ & $3^{-2}39^{2}$ & $0$ & $q^{-3} + 72q^{-1} + 1800 + \mathcal{O}(q)$ \\
$42$ & $3^{4}6^{-4}21^{-4}42^{4}$ & $0$ & $q^{-3} + 144q^{-1} + 1104 + \mathcal{O}(q)$ \\
$52$ & $1^{-2}2^{2}4^{-2}13^{2}26^{-2}52^{2}$ & $0$ & $q^{-3} + 72q^{-1} + 1800 + \mathcal{O}(q)$ \\
$52$ & $1^{2}4^{-2}13^{-2}52^2$ & $0$ & $q^{-3} + 72q^{-1} + 1560 + \mathcal{O}(q)$ \\
$70$ & $1^{-3}2^{3}5^{3}7^{3}10^{-3}14^{-3}35^{-3}70^{3}$ & $0$ & $q^{-3} + 108q^{-1} + 1512 + \mathcal{O}(q)$ \\
$70$ & $1^{2}2^{-2}5^{2}7^{-2}10^{-2}14^{2}35^{-2}70^{2}$ & $0$ & $q^{-3} + 96q^{-1} + 912 + \mathcal{O}(q)$ \\
$78$ & $3^{2}6^{-2}39^{-2}78^{2}$ & $0$ & $q^{-3} + 72q^{-1} + 456 + \mathcal{O}(q)$ \\
$91$ & $1^{1}7^{-1}13^{-1}91^{1}$ & $0$ & $q^{-3} + 36q^{-1} + 984 + \mathcal{O}(q)$ \\
$182$ & $1^{-1}2^{1}7^{1}13^{1}14^{-1}26^{-1}91^{-1}182^{1}$ & $0$ & $q^{-3} + 36q^{-1} + 408 + \mathcal{O}(q)$ \\ \bottomrule
\end{longtable} 
\end{center}

\FloatBarrier

\appendix
\section{Lattice automorphisms and cycle type}\label{app:cycletype}

\begin{defn}\label{def:cyc}
	A cycle type $C$ of order $n$ is a set of pairs $\{(t,b_{t})\}$, such that $t\mid n$, $b_{t} \in \Z$ and $\text{gcd}(\{t\}) = n$.
	As a shorthand we will write
	\[
	C := \prod_{t|n}t^{b_t}.
	\]
\end{defn}
Let $g$ be an automorphism of an integral lattice of order $n$. Then the characteristic polynomial of $g$ has integer coefficients and its roots are $n$-th roots of unity.
Such a polynomial is a product of cyclotomic polynomials
\be
\chi_g(q) = \prod_{t|n}\Phi_t(q)^{n_t}\,
\ee
where $\Phi_t$ is the $t$-th cyclotomic polynomial and $n_t \in \N$.
Using the M\"obius inversion formula
\[
 \Phi_t(q) = \prod_{d|t}(q^d - 1)^{\mu\big(\frac{t}{d}\big)},
\]
where $\mu$ is the M\"obius function, we can express the characteristic polynomial of $g$ as
\be
  \chi_g(q) = \prod_{t|n}(q^t-1)^{b_t}\,
\ee
with $b_t \in \Z$.
From this we define the cycle type of $g$ to be $\prod_{t\mid n}t^{b_t}$.

\section{Eta-quotients}\label{app:etaquotient}

\begin{defn}\label{def:eta}
	The Dedekind eta function, $\eta(\tau)$, is defined by the infinite product
	\[
	\eta(\tau) := q^{1/24}\prod_{n=1}^\infty(1-q^n).
	\]
\end{defn}

The modular transformation properties of the Dedekind eta-function are well known to be:
\begin{thm}[\cite{MR1474964}]\label{thm:eta} Under elements of $SL(2,\Z)$ $\eta(\tau)$ transforms as a modular form of weight $1/2$ with multiplier system $\vartheta(\gamma)$:
	\[
	\eta(\gamma \tau) = \vartheta(\gamma)(c\tau + d)^\frac{1}{2}\eta(\tau), \text{where } \gamma = \begin{pmatrix} a & b \\ c & d \end{pmatrix} \in SL(2,\Z),
	\]
	where
	\[
	\vartheta(\gamma) = \begin{cases} e(b/24) &\text{if } c = 0 \\ e\bigg(\frac{a+d-3c}{24c} - \frac{1}{2}s(d,c)\bigg) &\text{if } c > 0 \end{cases}
	\]
	with $s(d,c)$ the Dedekind sum
	\[
	s(d,c) = \sum_{0 \le n < c}\frac{n}{c}\bigg(\frac{dn}{c} - \bigg[\frac{dn}{c}\bigg] - \frac{1}{2}\bigg).
	\]
\end{thm}

\begin{defn}\label{def:etaquo}
	We define the eta-quotient $\eta_C(\tau)$ of cycle type $C$
	\[
	\eta_C(\tau) := \prod_{t \mid n} \eta(t\tau)^{b_t}
	\]
\end{defn}
To obtain the $SL(2,\Z)$-transformation of such an eta-quotient we can use need the following lemma:

\begin{prop}\label{prop:etatrans}
	For every $\begin{pmatrix} A & B \\ C & D \end{pmatrix} \in SL(2,\Z)$ and $q \in \Z$ there exist an $SL(2,\Z)$-transformation $\begin{pmatrix} a & b \\ c & d \end{pmatrix} \in SL(2,\Z)$ and three integers 
	$\alpha,\beta,\gamma \in \Z$ such that:  
	\begin{equation}\label{eq:etatrans}
	\eta\bigg(q\begin{pmatrix} A & B \\ C & D \end{pmatrix}(\tau)\bigg) = \eta\bigg(\begin{pmatrix} a & b \\ c & d \end{pmatrix}\bigg(\frac{\alpha\tau+\beta}{\gamma}\bigg)\bigg).
	\end{equation}
	
	Then
	
	\begin{enumerate}
		\item \label{eq:etatransa} $\alpha = (qA,C)$
		\item \label{eq:etatransb} $\gamma = \frac{q}{\alpha}$
		\item \label{eq:etatransc} $a = \frac{qA}{\alpha}$
		\item \label{eq:etatransd} $c = \frac{C}{\alpha}$
		\item \label{eq:etatranse} $ad \equiv 1 \pmod{c}$
		\item \label{eq:etatransf} $b = \frac{ad-1}{c}$
		\item \label{eq:etatransg} $\beta = qBd-Db$
	\end{enumerate}

	where $(x,y) \equiv \text{gcd}(x,y)$. 
\end{prop}
Note that relation \ref{eq:etatranse} implies that $d$ is the modular inverse of $a$ modulo $c$ which always exists as $(a,c) = 1$.

\begin{proof}
	Expanding the argument on both sides of equation \ref{eq:etatrans} we obtain
	\[
	\frac{qA\tau + qB}{C\tau+D} = \frac{a\alpha\tau + a\beta + b\gamma}{c\alpha\tau + c\beta + d\gamma}
	\]
	Equating term by term and solving the resulting relations gives the above result.
\end{proof}

Together with Theorem \ref{thm:eta} this proposition fully determines the modular transformation properties of the eta-quotients \ref{def:etaquo}.
\begin{cor}
	Note that this result in particular implies that
	\[
	c\frac{\alpha\tau+\beta}{\gamma} + d = \frac{C\tau+D}{\gamma}.
	\]
\end{cor}

By the following theorem we can express a wide range of functions in terms of eta-quotient and therefore determine their
modular transformation properties using Proposition \ref{prop:etatrans}. (Reference to Ono+ here)

\begin{thm}[\cite{MR2020489}]\label{thm:quomod}
	An eta-quotient $\eta_C(\tau)$ satisfies
	\[
	 \eta_C\bigg(\frac{a\tau + b}{c\tau + d}\bigg) = \chi(d)(c\tau + d)^k \eta_C(\tau),
	\]
	for every $\begin{pmatrix} a & b \\ c & d \end{pmatrix} \in \Gamma_0(n)$ with $k = (1/2) \sum_{t \mid n} b_t$
	and character $\chi(d) = \big(\frac{(-1)^k s}{d}\big)$, where $s := \prod_{t \mid n} t^{b_t}$,
	if the following additional conditions are satisfied by the cycle type $C$:
	\begin{enumerate}
		\item \label{cond1} $ 2k = \sum_{t \mid n} b_t \equiv 0 \pmod{2} $
		\item \label{cond2} $ \sum_{t \mid n}t b_t \equiv 0 \pmod{24} $
		\item \label{cond3} $ \sum_{t \mid n} \frac{n}{t}b_t \equiv 0 \pmod{24} $
	\end{enumerate}
	The order of vanishing of $\eta_C(\tau)$ at the cusp at $\frac{c}{d}$ is given by \\
	
	$$ \frac{n}{24}\sum_{t \mid n} \frac{(d,t)^2 b_t}{(d,\frac{n}{d}) dt}.$$
\end{thm}

\begin{cor}\label{cor:S_eta}
	For an eta-quotient satisfying the conditions of Theorem \ref{thm:quomod} the integer $\frac{1}{24}\sum_{t \mid n}t b_t$ gives the lowest non-trivial order in its $q$-expansion
	and the integer $\sum_{t \mid n} \frac{n}{t}b_t$ gives the lowest non-trivial order in the $q$-expansion of its $S$-transformation.
\end{cor}

\begin{defn}\label{def:cycpow}
	We define the $k^{\text{th}}$ power of $C$
	\[
	C^k := \prod_{t \mid n} \bigg(\frac{t}{(t,k)}\bigg)^{(t,k)b_t}.
	\]
\end{defn}

In particular the usual relations for the exponential hold:
\begin{lem}\label{lem:cycpow}
	\begin{equation*}
	(C^k)^l = C^{kl}
	\end{equation*}
\end{lem}

\begin{proof}
	The result follows immediately from $\big(\frac{t}{(t,k)},l\big)(t,k) = (t,kl)$.
\end{proof}

Next we will show that the conditions from Theorem \ref{thm:quomod} carry over when we take the power of a cycle type.
\begin{thm}\label{thm:powmod}
	If the conditions \ref{cond1} to \ref{cond3} from Theorem \ref{thm:quomod} hold for a cycle type $C$, then analogous properties hold for the cycle type $C^k$ so that
	\begin{enumerate}
		\item \[ \sum_{t \mid n} (t,k)b_t \equiv 0 \pmod{2} \]
		\item \[ \sum_{t \mid n}\frac{t}{(t,k)} (t,k)b_t = \sum_{t \mid n}t b_t \equiv 0 \pmod{24} \]
		\item \[ \sum_{t \mid n} \frac{n}{(n,k)}\frac{(t,k)^2b_t}{t} \equiv 0 \pmod{24} \]
	\end{enumerate}
\end{thm}

\begin{proof}
	If $(n,k) = 1$ there is nothing to prove. Hence by Lemma \ref{lem:cycpow} it is sufficient to prove our result for a prime divisor $p$ of $n$.
	We express the set of divisors of $n$ as the disjoint union of two sets $R$ and $S$ such that $p$ divides the elements of $S$ and is coprime to the elements of $R$.
	\begin{enumerate}
		\item \[
		\sum_{t \mid n} (t,p)b_t = \sum_{r \in R} b_r + p\sum_{s \in S} b_s = \sum_{r \in R} b_r + p\big(2k - \sum_{r \in R} b_r\big) = 2pk - (p-1)\sum_{r \in R} b_r
		\]
		
		If $p \neq 2$ then $2 \mid (p-1)$ so that this concludes the first part of the proof.
		
		If $p=2$ we need to show that $\sum_{r \in R} b_r \equiv 0 \pmod{2}$.
		As $\sum_{s \in S} s b_s \equiv 0 \pmod{2}$ Condition \ref{cond2} implies that $\sum_{r \in R} r b_r \equiv 0 \pmod{2}$. But as all the $r$ in the sum are odd
		this implies the desired result. \qed
		
		\item There is nothing to prove.
		
		\item We split the sum again into sums over multiples and coprimes of $p$:
		\begin{align*}
		\sum_{t \mid n} \frac{n}{(n,k)}\frac{(t,k)^2b_t}{t} &= \frac{1}{p}\sum_{r \in R} \frac{nb_r}{r} + p\sum_{s \in S} \frac{nb_s}{s}
		=\frac{1}{p}\sum_{r \in R} \frac{nb_r}{r} + p(24m - \sum_{r \in R} \frac{nb_r}{r}) \\
		&= 24pm - \frac{p^2-1}{p}\sum_{r \in R} \frac{nb_r}{r}
		\end{align*}
		
		If $p > 3$ then $24 \mid p^2 - 1$ so this concludes the proof.
		
		If $p = 3$ then $p^2 - 1 = 8$ so that we need to prove that $\sum_{r \in R} \frac{nb_r}{r} \equiv 0 \pmod{9}$.
		Condition \ref{cond2} implies that $\sum_{r \in R} nb_r r \equiv 0 \pmod{9}$ as $\sum_{s \in S} sb_s \equiv 0 \pmod{3}$. Therefore
		\[
		\sum_{r \in R}\frac{nb_r}{r} - nb_r r = - \sum_{r \in R} nb_r\frac{r^2-1}{r} \equiv 0 \pmod{9} \implies \frac{nb_r}{r} \equiv 0 \pmod{9},
		\]
		because $3 \mid (r^2-1), \forall r \in R$.
		
		If $p = 2$ then $p^2 -1 = 3$ so we need to prove that $\sum_{r \in R} \frac{nb_r}{r} \equiv 0 \pmod{16}$.
		Now
		\begin{equation}\label{eq:mod2}
		\sum_{r \in R}\frac{nb_r}{r} - nb_r r = - \sum_{r \in R} nb_r\frac{r^2-1}{r} \equiv 0 \pmod{16}.
		\end{equation}
		
		If $8 \mid n$ Condition \ref{cond2} implies that $\sum_{r \in R} nb_r r \equiv 0 \pmod{16}$ and the proof is concluded.
		
		If $4 \mid n$ Condition \ref{cond2} implies that $\sum_{r \in R} nb_r r \equiv 0 \pmod{8}$.
		It follows that $\sum_{r \in R} b_r \equiv 0 \pmod{2}$ and hence Condition \ref{cond3} implies that $\sum_{s \in S} b_s \equiv 0 \pmod{2}$.
		Then $\sum_{s \in S} nsb_s \equiv 0 \pmod{16}$ and by applying Condition \ref{cond2} again, we find that  $\sum_{r \in R} nb_r r \equiv 0 \pmod{16}$ which concludes the proof.
		
		If $2 \mid n$ we find by a similar reasoning that $\sum_{r \in R} nb_r r \equiv 0 \pmod{8}$ and $\sum_{s \in S} b_s \equiv 0 \pmod{2}$.
		From Condition \ref{cond3} we obtain that $\sum_{s \in S} b_s \equiv 0 \pmod{2}$
		Together with Conditions \ref{cond3} this implies that $\frac{nb_r}{r} \equiv 0 \pmod{4}$ and $\frac{nb_s}{s/(2,s)} \equiv 0 \pmod{8}$.
		This implies that $\sum_{s \in S} nsb_s \equiv 0 \pmod{16}$ and hence by Condition \ref{cond2} that $\sum_{r \in R} nb_r r \equiv 0 \pmod{16}$ which concludes the proof.

	\end{enumerate}
	
\end{proof}

\begin{cor}
	If $\eta_C(\tau)$ is modular form for the congruence subgroup $\Gamma_0(n)$, then $\eta_{C^k}$ is a modular form for the group $\Gamma_0(n/(n,k))$.
\end{cor}
 
 Table \ref{t:etabasis} shows bases of eta quotients for all the spaces of modular forms relevant to our computation.
\begin{table}
\begin{center}
	\begin{longtable}{llll}
$(N,k,\chi(d))$ & $\text{Dim}$ & $C_g$  \\ \midrule \endhead
$(4,12,1)$ & $7$ & $[2^{-24}4^{48}], [1^{8}2^{-24}4^{40}], [1^{16}2^{-24}4^{32}], [1^{24}2^{-24}4^{24}], [1^{32}2^{-24}4^{16}], [1^{40}2^{-24}4^{8}], [1^{48}2^{-24}]$ \\
$(6,4,1)$ & $5$ & $[1^{4}2^{-8}3^{-12}6^{24}], [1^{9}2^{-9}3^{-11}6^{19}], [1^{14}2^{-10}3^{-10}6^{14}], [1^{19}2^{-11}3^{-9}6^{9}], [1^{24}2^{-12}3^{-8}6^{4}]$ \\
$(6,8,1)$ & $9$ & $[1^{8}2^{-16}3^{-24}6^{48}], [1^{13}2^{-17}3^{-23}6^{43}], [1^{18}2^{-18}3^{-22}6^{38}], [1^{23}2^{-19}3^{-21}6^{33}], [1^{28}2^{-20}3^{-20}6^{28}],$ \\
\quad & ${}$ & $[1^{33}2^{-21}3^{-19}6^{23}], [1^{38}2^{-22}3^{-18}6^{18}], [1^{43}2^{-23}3^{-17}6^{13}], [1^{48}2^{-24}3^{-16}6^{8}]$ \\
$(6,12,1)$ & $13$ & $[1^{12}2^{-24}3^{-36}6^{72}], [1^{17}2^{-25}3^{-35}6^{67}], [1^{22}2^{-26}3^{-34}6^{62}], [1^{27}2^{-27}3^{-33}6^{57}], [1^{32}2^{-28}3^{-32}6^{52}],$\\
\quad & ${}$ & $[1^{37}2^{-29}3^{-31}6^{47}], [1^{42}2^{-30}3^{-30}6^{42}], [1^{47}2^{-31}3^{-29}6^{37}], [1^{52}2^{-32}3^{-28}6^{32}],$\\
\quad & ${}$ & $[1^{57}2^{-33}3^{-27}6^{27}], [1^{62}2^{-34}3^{-26}6^{22}], [1^{67}2^{-35}3^{-25}6^{17}], [1^{72}2^{-36}3^{-24}6^{12}]$\\
$(8,6,1)$ & $7$ & $[4^{-12}8^{24}], [1^{-4}2^{10}4^{-14}8^{20}], [1^{-8}2^{20}4^{-16}8^{16}], [1^{-12}2^{30}4^{-18}8^{12}],$ \\
\quad & ${}$ & $[1^{-16}2^{40}4^{-20}8^{8}], [1^{-20}2^{50}4^{-22}8^{4}], [1^{-24}2^{60}4^{-24}]$ \\
$(9,8,1)$ & $9$ & $[1^{24}3^{-8}], [3^{-8}9^{24}], [1^{15}3^{4}9^{-3}], [1^{12}3^{4}], [1^{6}3^{16}9^{-6}],$\\
\quad & ${}$ & $[1^{3}3^{16}9^{-3}], [1^{-3}3^{28}9^{-9}], [1^{-6}3^{28}9^{-6}], [1^{-12}3^{40}9^{-12}]$\\
$(10,4,1)$ & $7$ &  $[1^{-10}2^{20}5^{2}10^{-4}], [1^{-5}2^{15}5^{1}10^{-3}], [2^{10}10^{-2}], [1^{5}2^{5}5^{-1}10^{-1}]$ \\
\quad & ${}$ & $[1^{10}5^{-2}], [1^{15}2^{-5}5^{-3}10^{1}], [1^{20}2^{-10}5^{-4}10^{2}]$ \\
$(10,8,1)$ & $13$ & $[1^{-20}2^{40}5^{4}10^{-8}], [1^{-15}2^{35}5^{3}10^{-7}], [1^{-10}2^{30}5^{2}10^{-6}], [1^{-5}2^{25}5^{1}10^{-5}], [2^{20}10^{-4}],$ \\
\quad & ${}$ & $[1^{5}2^{15}5^{-1}10^{-3}], [1^{10}2^{10}5^{-2}10^{-2}], [1^{15}2^{5}5^{-3}10^{-1}], [1^{20}5^{-4}],$\\
\quad & ${}$ & $[1^{25}2^{-5}5^{-5}10^{1}], [1^{30}2^{-10}5^{-6}10^{2}], [1^{35}2^{-15}5^{-7}10^{3}], [1^{40}2^{-20}5^{-8}10^{4}]$\\
$(12,4,1)$ & $9$ & $[1^{-4}2^{10}3^{-12}4^{-4}6^{30}12^{-12}], [3^{-16}6^{40}12^{-16}], [1^{4}2^{-10}3^{-20}4^{4}6^{50}12^{-20}], [1^{8}2^{-20}3^{-24}4^{8}6^{60}12^{-24}],$ \\
\quad & ${}$ & $[2^{-4}4^{8}6^{-4}12^{8}], [2^{-2}4^{4}6^{-6}12^{12}], [6^{-8}12^{16}], [2^{2}4^{-4}6^{-10}12^{20}], [2^{4}4^{-8}6^{-12}12^{24}]$\\
$(15,4,1)$ & $8$ & $[1^{8}3^{-4}5^{-4}15^{8}], [1^{-3}3^{1}5^{15}15^{-5}], [1^{-2}5^{10}], [1^{-1}3^{-1}5^{5}15^{5}],$ \\
\quad & ${}$ & $[3^{-2}15^{10}], [1^{5}3^{5}5^{-1}15^{-1}], [1^{10}5^{-2}], [1^{15}3^{-5}5^{-3}15^{1}]$ \\
$(18,4,1)$ & $13$ & $[1^{-12}2^{24}3^{4}6^{-8}], [1^{-9}2^{21}3^{3}6^{-7}], [1^{-6}2^{18}3^{2}6^{-6}], [1^{-3}2^{15}3^{1}6^{-5}], [2^{12}6^{-4}],$\\
\quad & ${}$ & $[1^{3}2^{9}3^{-1}6^{-3}], [1^{6}2^{6}3^{-2}6^{-2}], [1^{9}2^{3}3^{-3}6^{-1}], [1^{12}3^{-4}],$\\
\quad & ${}$ & $[1^{15}2^{-3}3^{-5}6^{1}], [1^{18}2^{-6}3^{-6}6^{2}], [1^{21}2^{-9}3^{-7}6^{3}], [1^{24}2^{-12}3^{-8}6^{4}]$\\
$(20,4,1)$ & $12$ & $[2^{2}4^{-4}10^{-10}20^{20}], [1^{-13}2^{29}4^{-10}5^{9}10^{-9}20^{2}], [1^{-20}2^{50}4^{-20}5^{4}10^{-10}20^{4}], [1^{7}2^{-3}4^{-2}5^{-3}10^{-1}20^{10}],$ \\
\quad & ${}$ & $[10^{-8}20^{16}], [1^{-16}2^{40}4^{-16}], [1^{2}2^{-1}4^{-1}5^{-2}10^{-3}20^{13}], [1^{5}2^{-2}4^{-1}5^{-1}10^{-6}20^{13}],$ \\
\quad & ${}$ & $[1^{-14}2^{37}4^{-15}5^{6}10^{-9}20^{3}], [1^{-14}2^{33}4^{-13}5^{6}10^{-5}20^{1}], [1^{-8}2^{20}4^{-8}5^{8}10^{-4}], [1^{-12}2^{26}4^{-10}5^{12}10^{-10}20^{2}]$ \\
$(44,2,1)$ & $9$ & $[1^{8}2^{-4}], [2^{-4}4^{8}], [22^{-4}44^{8}], [1^{2}11^{2}], [1^{1}2^{-3}4^{4}11^{-3}22^{9}44^{-4}],$ \\
\quad & ${}$ & $[1^{-1}2^{2}4^{-1}11^{3}22^{-2}44^{3}], [1^{3}2^{-2}4^{3}11^{-1}22^{2}44^{-1}], [1^{-3}2^{7}11^{1}22^{-1}], [1^{-3}2^{9}4^{-4}11^{1}22^{-3}44^{4}]$ \\
$(92,1,\big(\frac{-23}{d}\big))$ & $6$ & $[1^{1}23^{1}], [2^{1}46^{1}], [4^{1}92^{1}], [1^{2}2^{-1}23^{2}46^{-1}], [1^{1}2^{-1}4^{-1}23^{1}46^{-1}92^{1}], [2^{-1}4^{2}46^{-1}92^{2}]$ \\ \bottomrule	\end{longtable} 
\end{center}
\caption{Eta quotient bases  \label{t:etabasis}}
\end{table}

\FloatBarrier
 
\section{Lattice theta-functions and the inversion formula}\label{app:inverse}
Let $L$ be a lattice, $L'$ its dual and $\lambda$ and $\beta$ be vectors.
Define
\[
 \theta_{L+\lambda}(\tau) = \sum_{\alpha \in L} e\bigg(\frac{\langle \alpha+\lambda, \alpha+\lambda \rangle}{2}\tau\bigg)
\]
and
\[
 \theta^{\beta}_L \sum_{\alpha \in L} e\bigg(\frac{\langle \alpha, \alpha \rangle}{2}\tau + \langle\alpha,\beta\rangle\bigg).
\]

Then the following result holds
\begin{thm}[Inversion Formula, \cite{MR1474964}]\label{thm:inv}
 \[
  \theta_L^{\beta}\big(\frac{-1}{\tau}\big) = (\det(L))^{-\frac{1}{2}}(-iz)^{\mathrm{Rank}(L)}\theta_{L'+\beta}(\tau)
 \]
\end{thm}

\medskip
 
\bibliographystyle{alpha}
\bibliography{./refmain}

\end{document}